\documentclass[a4paper,12pt]{article}

\usepackage{pstricks}
\usepackage{graphicx,psfrag}
\usepackage{amsmath}
\usepackage{amssymb}
\usepackage{amsthm}
\usepackage{color}

                \newcommand {\Om}  {\Omega}
      \newcommand {\pl}   {\partial}        \newcommand {\s}    {\sigma}

      \newcommand {\RRR}  {{\mathbb R}}

     \newcommand {\beq}  {\begin{equation}}
      \newcommand {\eeq}  {\end{equation}}

      \newtheorem{theorem}{Theorem}

      \newtheorem{zam}{Remark}
      \newtheorem{opr}{Definition}

\author{Alexander Plakhov\thanks{Aveiro University, Portugal} \and Vera Roshchina\thanks{\'{E}vora University,  Portugal}}

\title{Invisibility in billiards}

\date{}

\begin{document}

\maketitle

\begin{abstract}
{The question of invisibility for bodies with mirror surface is studied in the framework of geometrical optics}. We construct bodies that are invisible/have zero resistance in two {mutually orthogonal} directions, and prove that there do not exist bodies which are invisible/have zero resistance in all possible directions {of incidence}.
\end{abstract}

\begin{quote}
\end{quote}

\begin{quote}
\end{quote}

\section{Introduction}

Problems related to constructing invisible bodies are in the focus of attention nowadays. Apart from having the potential for various applications such as constructing invisible submarines, creating improved lenses for DVD readers that would allow to read denser information, the topic attracts attention of general public, mostly due to the concept of invisible cloak, which is a popular topic in fiction and movies (see \cite{NYTimes2007}, \cite{CNN}, \cite{DailyMail}). Apart from various implementations of such a cloak, which use cameras to project the image from behind on a specially designed surface (e.g. see \cite{UTokyo}), the bulk of the studies on invisible bodies focus on constructing materials with special refractive properties. The research in this direction was pioneered by V.G.Veselago in 1960's, who published a theoretical study of materials that allowed for a negative refractive index \cite{Veselago}. Such materials do not exist in nature, however, {they} can be engineered. They are called metamaterials and have been successfully constructed for some limited settings. Following \cite{CloakSchurig}, researchers at Duke University have demonstrated a body invisible for microwaves (see \cite{NYTimes2007}); in \cite{CloakValentine} construction of a 3D optical metamaterial with a negative refraction index was reported; in \cite{CloakErgin} metamaterials are used to hide a bump in a metallic mirror from angles up to 60$^\circ$ and a large bandwidth of unpolarized light.

In this article we are concerned with invisibility in billiards. We consider bodies with mirror surface and light rays falling on it. Invisibility in a direction $v \in S^2$ means that any incident light ray which initially moves along a straight line in this direction, after several reflections from the body's surface will eventually move along the same straight line. Invisibility in a set of directions means that the above is true for any direction from this set. In \cite{0-resist} { the notion of billiard invisibility was introduced and some examples of} bodies invisible in one direction { were provided}. { In this article we continue the study of this topic; our results are} twofold. First, we show that there exist bodies invisible in two { mutually orthogonal} directions. Second, we prove that bodies invisible in all directions in $S^2$ do not exist.

{ Notice that somewhat similar results were obtained in wave scattering. It was shown, in the first Born approximation, that there exist bodies invisible for any {\it finite} number of directions \cite{Dev}, and there are no bodies invisible for {\it all} directions of incidence \cite{Wolf}.}

There is a closely related sort of problems. Consider a parallel flow of point particles at a velocity $v \in S^2$ falling on a body $B$ at rest. The flow is so rarefied that the particles do not mutually interact. Particles reflect elastically when colliding with the body surface and move freely between consecutive collisions. The problem of minimal resistance going back to Newton \cite{N} consists in { finding} a body, from a given class of bodies, that experiences the smallest possible force of flow pressure, or resistance force. Since 1990's, many interesting results in this problem have  been obtained by various authors (see, e.g., \cite{BB,BK,CL1,LO,LP1,RMS_review,ESAIM}). A body of zero resistance has been { provided} in \cite{0-resist}; this means that the final velocity of any particle incident on the body coincides with the velocity of incidence $v$.

In this paper a body having zero resistance in two directions is constructed, and it is proved that bodies having zero resistance in all directions do not exist. Notice that invisibility implies zero resistance, therefore any body invisible in two directions will have zero resistance in these directions, and impossibility of zero resistance in all directions implies that invisibility in all directions is also impossible.

There are many questions still open. { Do there exist} bodies invisible/having zero resistance in three or more directions, or even in a set of directions of positive measure? We suppose that the answer to the last part of the question is negative, but cannot prove it.

We start with exact definitions.

\begin{opr}\label{o 1}
A body \rm is a bounded set with a piecewise smooth boundary in $\RRR^3$.
\end{opr}

Consider the billiard in $\RRR^3 \setminus B$, and take a convex body $C$ containing $B$. For { a regular point of the boundary} $\xi \in \pl C$, denote by $n(\xi)$ the unit outer normal to $\pl C$ at $\xi$. Introduce the measurable spaces $(\pl C \times S^2)_\pm := \{ (\xi,v) \in \pl C \times S^2 :\, \pm n(\xi) \cdot v \ge 0 \}$ equipped with the measures $d\mu_{\pm}(\xi,v) = \pm (n(\xi) \cdot v)\, d\xi\, dv$, correspondingly, where dot means scalar product. { Note that the set of singular points of any convex set has zero Lebesgue measure, therefore the union $(\pl C \times S^2)_- \cup (\pl C \times S^2)_+$ is a full measure set in $\pl C \times S^2$.}

The motion of a billiard particle interacting with the body $B$ can be generally described as follows. First the particle moves freely with a velocity $v$, then intersects $\pl C$ at a point $\xi$ and moves in $C$ making reflections from $\pl B$, and finally, leaves $C$ at a point $\xi^+$ and moves freely with a velocity $v^+$ afterwards (see Fig.~\ref{fig scatt}). According to this description, a mapping $(\xi,v) \mapsto (\xi^+ = \xi^+_{B,C}(\xi,v),\, v^+ = v^+_{B,C}(\xi,v))$ is defined, which is a measure preserving one-to-one correspondence between full measure subsets of $(\pl C \times S^2)_-$ and $(\pl C \times S^2)_+$.

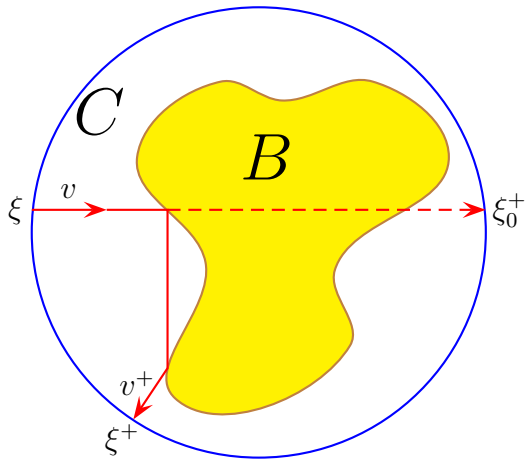
\begin{figure}[h]
\begin{picture}(0,175)
\rput(7.7,3){
\pscircle[linecolor=blue](0,0){3}
\psecurve[linecolor=brown,fillstyle=solid,fillcolor=yellow]
(1.2,-1.5)(1,-0.5)(2.5,1)(1.5,2)(0.3,1.75)(-0.5,2)(-1.6,1)(-0.7,-0.4)
(-1.1,-2.2)(1.2,-1.5)(1,-0.5)(2.5,1)
\psline[linecolor=red,arrows=->,arrowscale=2](-2.985,0.3)(-2,0.3)
\rput(-3.2,0.3){$\xi$}
\rput(3.3,0.3){$\xi^+_0$}
\rput(-1.8,-2.75){$\xi^+$}
\rput(-1.6,-2){$v^+$}
\rput(-2.5,0.55){$v$}
\rput(0.1,1){\Huge $B$}
\rput(-2.1,1.6){\Huge $C$}
\psline[linecolor=red,arrows=->,arrowscale=2](-2,0.3)(-1.2,0.3)(-1.2,-1.8)(-1.65,-2.48)
\psline[linecolor=red,linestyle=dashed,arrows=->,arrowscale=2](-1.2,0.3)(2.985,0.3)
}
\end{picture}
\caption{ The broken line with the endpoints $\xi$ and $\xi^+$ is a billiard trajectory in the complement of $B$. The straight line with the endpoints $\xi$ and $\xi_0^+$ is a trajectory corresponding to the case $B = \emptyset$.}
\label{fig scatt}
\end{figure}

Note that for a zero measure set of values $(\xi,v) \in  (\pl C \times S^2)_-$, the corresponding particle hits $\pl B$ at a singular point, or gets trapped in $C$, or makes infinitely many reflections in a finite time. For these values the mapping is not defined.

For future use we introduce the notation $\xi^+_0 := \xi^+_{\emptyset,C}$, corresponding to the case $B = \emptyset$ where all particles move freely inside $C$; see Fig.~\ref{fig scatt}.

\begin{opr}\label{o 2}

\rm (a) We say that the body $B$ {\it has zero resistance in the direction} $v$, if $v^+_{B,C}(\xi,v) = v$ for all $\xi$ (see Fig.
\ref{fig zerores invis} (a)).

(b) We say that the body $B$ {\it is invisible in the direction} $v$, if it has zero resistance in this direction and, additionally, $\xi^+_{B,C}(\xi,v) - \xi$ is parallel to $v$ (see Fig.~\ref{fig zerores invis} (b)).

(c) Let $A \subset S^2$. The body $B$ is said to be {\it invisible/have zero resistance in the set of directions} $A$, if it is invisible/has zero resistance in any direction $v \in A$.
\end{opr}

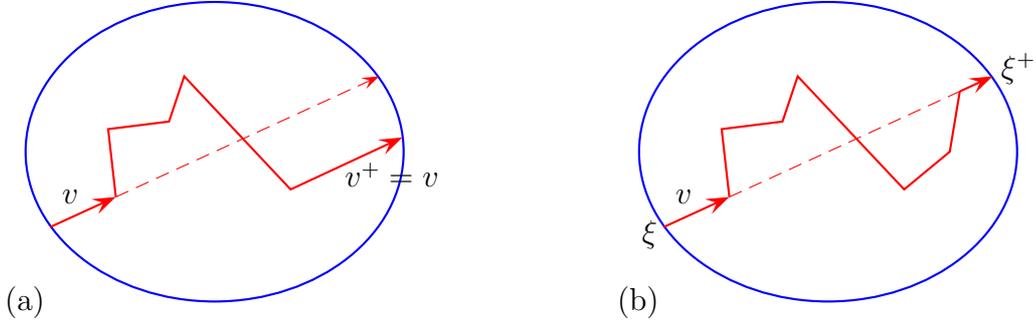
\begin{figure}[h]
\begin{picture}(0,160)
\rput(4,3){
\scalebox{1}{
\psellipse[linecolor=blue](0,0)(2.5,2)
\psline[linecolor=red,arrows=->,arrowscale=2](-2.165,-1)(-1.299,-0.6)
\psline[linecolor=red,linestyle=dashed,linewidth=0.4pt,arrows=->,arrowscale=2](-1.299,-0.6)(2.165,1)
\psline[linecolor=red,arrows=->,arrowscale=2](-1.299,-0.6)(-1.4,0.3)(-0.6,0.4)(-0.4,1)(1,-0.5)(2.47,0.19)
\rput(2.33,-0.3){$v^+=v$}
\rput(-1.9,-0.6){$v$}
\rput(-2.5,-2){(a)}
\rput(8,0){
\psellipse[linecolor=blue](0,0)(2.5,2)
\psline[linecolor=red,arrows=->,arrowscale=2](-2.165,-1)(-1.299,-0.6)
\psline[linecolor=red,arrows=->,arrowscale=2](1.732,0.8)(2.165,1)
\psline[linecolor=red,linestyle=dashed,linewidth=0.4pt](-1.299,-0.6)(1.732,0.8)
\psline[linecolor=red](-1.299,-0.6)(-1.4,0.3)(-0.6,0.4)(-0.4,1)(1,-0.5)(1.6,0)(1.732,0.8)
\rput(-1.9,-0.6){$v$}
\rput(-2.35,-1.1){$\xi$}
\rput(2.5,1.1){$\xi^+$}
\rput(-2.5,-2){(b)}
}
}}
\end{picture}
\caption{A typical billiard path in the case of a body (a) having zero resistance in the direction $v$; (b) invisible in the direction $v$. The body is not shown in both cases.}
\label{fig zerores invis}
\end{figure}

One easily sees that these definitions do not depend on the choice of the ambient body $C$.

{ The plan of the paper is as follows. In section 2 we construct bodies of zero resistance {\it in two directions} and bodies invisible {\it in two directions}. In section 3 we prove that bodies invisible {\it in all directions} and bodies of zero resistance {\it in all directions} do not exist.}

\section{Bodies invisible in two directions}

\begin{theorem}\label{t 2 directions}
For any two mutually perpendicular directions $v_1$ and $v_2 \in S^2$,

(a) there exists a body having zero resistance in both directions;

(b) there exists a body invisible in these directions.
\end{theorem}

\begin{proof}
We first construct a basic two-dimensional body and show that it has got zero resistance in one direction, and then extend the construction to three dimensions.

Take a plane $\Pi$ containing $v_1$ and perpendicular to $v_2$, and consider two parabolas in this plane with common axis parallel to $v_1$ and with common focus and centrally symmetric to each other with respect to the focus. Take two straight lines in the same plane parallel to the common axis of the parabolas and situated at the same distance on both sides of it. Next, consider two curvilinear triangles formed by segments of these straight lines and by arcs of the parabolas, see Fig.~\ref{fig parabolas} (a). The union of these triangles is a (disconnected) two-dimensional figure having zero resistance to a parallel flow on the plane falling at the velocity $v_1$. Indeed, taking into account the focal property of parabola, we see that any incident particle of the flow, after reflecting from a parabola, passes through the focus, then reflects from the other parabola, and moves afterwards with the velocity $v_1$. That is, a parallel flow with velocity $v_1$ is transformed into a parallel flow with the same velocity.

\begin{figure}[h]
\begin{picture}(0,160)
\rput(3.5,3){
\scalebox{1.2}{

\psframe[linewidth=0pt,linecolor=white,fillstyle=solid,fillcolor=blue](-2,-1.49)(2,1.49)
\psframe[linewidth=0pt,linecolor=white,fillstyle=solid,fillcolor=blue](-2,-1.49)(2,1.49)
\parabola[linewidth=0.4pt,linecolor=black,fillstyle=solid,fillcolor=white](2,1.5)(0,-0.5)
\parabola[linewidth=0.4pt,linecolor=black,fillstyle=solid,fillcolor=white](2,-1.5)(0,0.5)
\parabola[linewidth=0.8pt,linecolor=white](1,0)(0,0.5)
\parabola[linewidth=0.4pt,linecolor=black,linestyle=dashed](2,1.5)(0,-0.5)
\parabola[linewidth=0.4pt,linecolor=black,linestyle=dashed](2,-1.5)(0,0.5)
\psline[linewidth=0.4pt,linecolor=black](2,1.5)(2,-1.5)
\psline[linewidth=0.4pt,linecolor=black](-2,1.5)(-2,-1.5)
\psline[linewidth=0.4pt,linecolor=red,arrows=->,arrowscale=1.6](-1.8,2.3)(-1.8,1.5)
\psline[linewidth=0.4pt,linecolor=red,arrows=->,arrowscale=1.6](-1.2,2.3)(-1.2,1.5)
\psline[linewidth=0.4pt,linecolor=red,arrows=->,arrowscale=1.6](-1.8,1.5)(-1.8,1.12)(1.8,-1.12)(1.8,-2)
\psline[linewidth=0.4pt,linecolor=red,arrows=->,arrowscale=1.6](-1.2,1.5)(-1.2,0.22)(1.2,-0.22)(1.2,-2)
\psdots[dotsize=3pt](0,0)
\psline[linewidth=0.6pt,linecolor=red,arrows=->,arrowscale=1.6](-2.5,0.5)(-2.5,-0.5)
\rput(-2.8,0){$v_1$}
\rput(-2.4,-2){(a)}
\rput(8,0){
\psframe[linewidth=0pt,linecolor=white,fillstyle=solid,fillcolor=blue](-2,-1.49)(-1.6,1.49)
\psframe[linewidth=0pt,linecolor=white,fillstyle=solid,fillcolor=blue](1.6,-1.49)(2,1.49)
\parabola[linewidth=0.4pt,linecolor=black,fillstyle=solid,fillcolor=white](2,1.5)(0,-0.5)
\parabola[linewidth=0.4pt,linecolor=black,fillstyle=solid,fillcolor=white](2,-1.5)(0,0.5)
\parabola[linewidth=0.8pt,linecolor=white](1.6,-0.78)(0,0.5)
\parabola[linewidth=0.8pt,linecolor=white](1.6,0.78)(0,-0.5)
\parabola[linewidth=0.4pt,linecolor=black,linestyle=dashed](2,1.5)(0,-0.5)
\parabola[linewidth=0.4pt,linecolor=black,linestyle=dashed](2,-1.5)(0,0.5)
\psline[linewidth=0.4pt,linecolor=black](2,1.5)(2,-1.5)
\psline[linewidth=0.4pt,linecolor=black](-2,1.5)(-2,-1.5)
\psline[linewidth=0.4pt,linecolor=black](1.6,0.78)(1.6,-0.78)
\psline[linewidth=0.4pt,linecolor=black](-1.6,0.78)(-1.6,-0.78)
\rput(-2.4,-2){(b)}
}
}}
\end{picture}
\caption{Two-dimensional figures invisible in one direction: (a) A union of two triangles. (b) A union of two trapezia.}
\label{fig parabolas}
\end{figure}
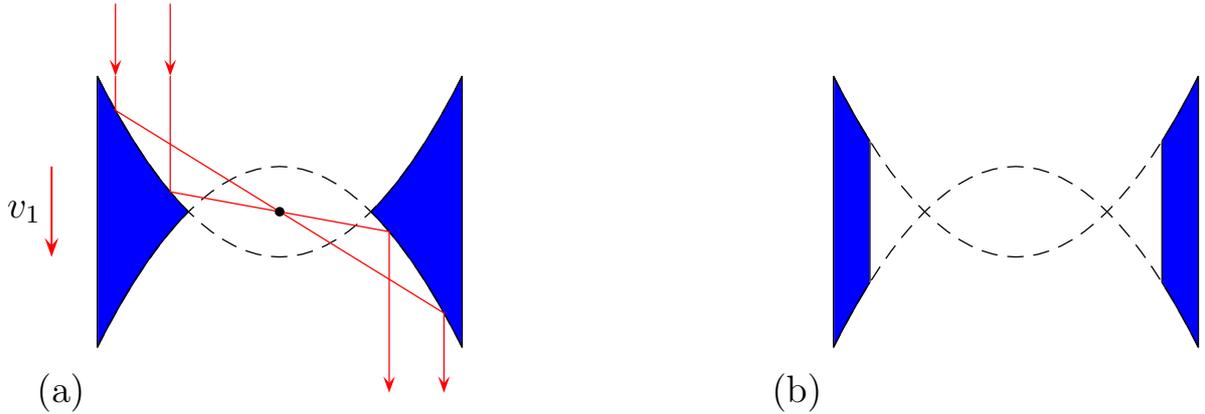

Note also that the union of two trapezia bounded by arcs of the parabolas and by two pairs of straight lines (where two lines in each pair are { parallel to the axis and} symmetric to each other with respect to { it}) is also a figure of zero resistance in the direction $v_1$ (see Fig.~\ref{fig parabolas} (b)).

If we choose a basis in the two-dimensional space $\Pi$ in a way that $v_1=(0,-1)$ { and the origin coincides with the focus}, the bodies described above are given by the following parametric family
$$
B(\alpha,\beta,\gamma) = \left\{ (x,y)\,\Bigl|\, |y|\leq \alpha x^2-\frac{1}{4\alpha}, \;\gamma\leq |x|\leq \beta \right\},
$$
where the { positive} parameters $\alpha,\beta$ and $\gamma$ are such that $2\alpha\beta>1$ and $\gamma <\beta$. The ``triangular'' construction then corresponds to $\gamma\leq \frac{1}{2\alpha}$, and the ``trapezial'' one -- to $\gamma>\frac{1}{2\alpha}$.

Then we obtain a three-dimensional body $B_1$ invisible in the same direction $v_1$ by parallel translation of the two-dimensional figure of Fig.~\ref{fig parabolas} (a) in the direction $v_2$ orthogonal to the plane of the figure (see Fig.~\ref{fig:B1B2} (a)). The length $h$ of this translation is equal to the height of the figure (that is, to the length of the rectilinear side of a triangle). Then we construct another body $B_2$ by rotating $B_1$ by $\pi/2$ around { its symmetry axis} perpendicular to $v_1$ and $v_2$ (see Fig.~\ref{fig:B1B2} (b)). The resulting body $B_2$ has zero resistance in the direction $v_2$. Finally, we show that the body $B=B_1\cap B_2$ (see Fig.~\ref{fig:B1B2} (c)) has zero resistance in both directions $v_1$ and $v_2$.

\begin{figure}[h]
\centering
\includegraphics[height=160pt, keepaspectratio]{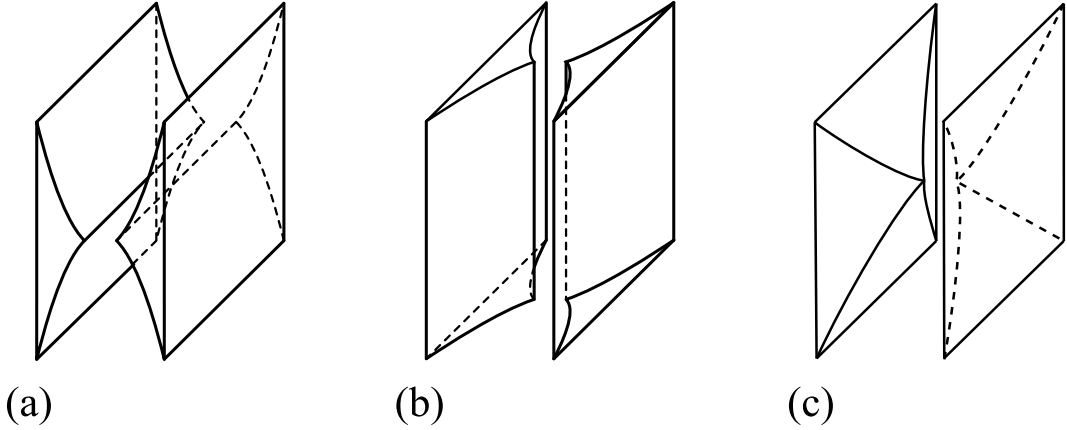}
\caption{Construction of a body of zero resistance in two directions}\label{fig:B1B2}
\end{figure}

Indeed, the intersection of $B$ with any plane parallel to $\Pi$ is a union of two curvilinear trapezia, besides { the} outer normal vector to $\pl B$ at { any point of a} curvilinear side of the trapezia is parallel to $\Pi$. Therefore any incident particle that initially moves in this plane with the velocity $v_1$, after two reflections from curvilinear sides of the trapezia will eventually move in the same plane and with the same velocity $v_1$. Therefore $B$ has zero resistance in the direction $v_1$. For $v_2$ the argument is the same.

Again, we can give a more rigorous algebraic representation of $B$. Choose a basis in $\RRR^3$ in a way that $v_1=(0,-1,0)$,\, $v_2=(0,0,-1)$, { and the origin coincides with the center of symmetry of the body}.
Then
$$
B(\alpha,\beta,\gamma) = \left\{ (x,y,z)\,\Bigl|\, |y|,|z|\leq \alpha x^2-\frac{1}{4\alpha}, \;\gamma\leq |x|\leq \beta \right\},
$$
where the parameters $\alpha,\beta$ and $\gamma$ are the same as before.

To obtain a body invisible in the directions { $v_1$ and $v_2$}, it suffices to take a union of 4 identical bodies obtained from $B$ by shifts by 0, $hv_1, \ hv_2$, and $hv_1 + hv_2$ (see Fig.~\ref{fig:Invis}).

\begin{figure}[h]
\centering
\includegraphics[height=160pt, keepaspectratio]{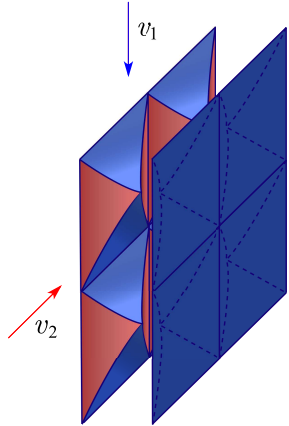}
\caption{A body invisible in two directions}\label{fig:Invis}
\end{figure}

\end{proof}

\section{Non-existence of bodies invisible in all directions}

\begin{theorem}\label{t all directions}
There do not exist bodies that

(a) are invisible in all directions;

(b) have zero resistance in all directions.
\end{theorem}

\begin{proof}
Let us first outline the idea of proof. { Note that statement (b) of the theorem implies statement (a), but for methodological reasons we first prove (a), and then (b).}

The phase space of the billiard in $C\setminus B$ is $(C \setminus B) \times S^2$, with the coordinate $(x,v)$ and the element of Liouville phase volume $dx\, dv$. Taking into account that the area of unit sphere is $|S^2| = 4\pi$, we get that the volume of phase space equals $4\pi |C \setminus B|$.

The phase volume can be estimated in a different way. Summing up the lengths of all billiard trajectories { (of course summation amounts to integration over the initial data)}, we get the volume of the reachable part of the phase space. Comparing the case of an invisible body $B$ (assuming that such a body exists) with the case $B = \emptyset$ (where the body is absent) and comparing the lengths of trajectories with identical initial data, we see that the length of the trajectory in the first case is always greater or equal than in the second one, therefore the phase volume is also greater in the first case, $4\pi |C \setminus B| \ge 4\pi |C|$. This is a contradiction.

The case of the body $B$ of zero resistance is a little bit more complicated. We also compare it with the case $B = \emptyset$ and show that the sum of the lengths of billiard trajectories with the fixed initial velocity in the first case is greater than in the second one. Then, summing up over all initial velocities, { again} we come to the conclusion that the phase volume in the first case is greater or equal than in the second one.

Let us pass to a more precise exposition. Suppose a billiard particle starts the motion at a point $\xi \in \pl C$ and with the initial velocity $v \in S^2$ turned inside $C$ (which means that $n(\xi) \cdot v \le 0$), and let $t \ge 0$; then assign the new coordinate $(\xi,v,t)$ to the point of phase space reached by the particle in the time $t$. The element of phase volume then takes the form $(-n(\xi) \cdot v)\, d\xi\, dv\, dt = d\mu_-(\xi,v)\, dt$. Further, denote by $\tau(\xi,v)$ the length of the particle's trajectory inside $C$, from the starting point $\xi$ until the point $\xi^+ = \xi_{B,C}^+(\xi,v)$ where it leaves $C$. Recall that $(\pl C \times S^2)_\pm = \{ (\xi,v) \in \pl C \times S^2 : \, \pm n(\xi) \cdot v \ge 0 \}$. Then the volume of the reachable part of phase space equals
$$
{ \int_{(\pl C \times S^2)_-} \int_0^{\tau(\xi,v)}\, dt\, d\mu_-(\xi,v) =} \int_{(\pl C \times S^2)_-} \tau(\xi,v)\, d\mu_-(\xi,v),
$$
{ Recall that $\xi = \xi_{B,C}^+(\xi,v)$.} Taking into account that the distance between the initial and final points of the trajectory does not exceed its length,
\beq\label{1}
|\xi^+ - \xi| \le \tau(\xi,v),
\eeq
and at some points { $(\xi,v)$ (and therefore in their neighborhoods)} the inequality in (\ref{1}) is strict,  we get
\beq\label{2}
\int_{(\pl C \times S^2)_-} |\xi^+ - \xi|\, d\mu_-(\xi,v) < 4\pi |C \setminus B|.
\eeq

Now let $\xi^+_0 = \xi^+_0(\xi,v)$ be the point where the particle leaves $C$ in the case $B = \emptyset$. In other words, $\xi^+_0$ is the point of intersection of the ray $\xi + vt,\, t > 0$ with $\pl C$. In this case all the phase space is reachable, besides one has equality in (\ref{1}), therefore in place of (\ref{2}) one gets the equality
\beq\label{3}
\int_{(\pl C \times S^2)_-} |\xi^+_0 - \xi|\, d\mu_-(\xi,v) = 4\pi |C|.
\eeq
If $B$ is invisible in all directions then $\xi^+_0 = \xi^+$, therefore from (\ref{2}) and (\ref{3}) one gets
$$
4\pi |C| < 4\pi |C \setminus B|,
$$
which is a contradiction.

Now let $B$ have zero resistance in all directions, that is, $v^+_{B,C}(\xi,v) = v$ for all $\xi$ and $v$. Denote by $\pl C_v^\pm$ the set of points $\xi$ such that $\pm n(\xi) \cdot v \ge 0$ with the induced measure $(\pm n(\xi) \cdot v)\, d\xi$. Since the mapping $(\xi,v) \mapsto (\xi_{B,C}^+(\xi,v),v)$ from $((\pl C \times S^2)_-, \ \mu_-)$ to $((\pl C \times S^2)_+, \ \mu_+)$ preserves the measure, we conclude that the induced mapping $\xi \mapsto \xi_{B,C}^+(\xi,v)$ from $\pl C_v^-$ to $\pl C_v^+$ preserves the induced measure for almost every $v$. Fix $v$ and introduce an orthogonal coordinate system $\xi_1,\, \xi_2,\, \xi_3$ in such a way that $v$ takes the form $v = (0,0,1)$. Then the subsets $\pl C_v^\pm$ take the form
$$
\pl C_v^\pm = \{ (\xi_1, \xi_2, \xi_3) : (\xi_1, \xi_2) \in \Om,\, \xi_3 = f^\pm(\xi_1,\xi_2) \},
$$
where $\Om$ is a convex domain in $\RRR^2$, $f^-$ is a convex function on $\Om$,\, $f^+$ is a concave function on $\Om$, and $f^- \le f^+$. Then both measures  $(\pm n(\xi) \cdot v)\, d\xi$ on $\pl C_v^\pm$ take the form $d\xi_1\, d\xi_2$, and the mapping $\xi \mapsto \xi_{B,C}^+(\xi,v)$ takes the form $(\xi_1, \xi_2, f^-(\xi_1,\xi_2)) \mapsto (\s(\xi_1,\xi_2), f^+(\s(\xi_1,\xi_2)))$, where $\s$ is a transformation of $\Om$ preserving the Lebesgue measure $d\xi_1\, d\xi_2$; see Fig.~\ref{fig Cavalieri}.

\begin{figure}[h]
\begin{picture}(0,210)
\rput(8,4){
\psecurve[linecolor=blue](3,-1.65)(3.8,0)(3,1.65)(2,2.2)(1,2.42)(0,2.5)(-1,2.42)(-2,2.2)(-3,1.65)(-3.8,0)
(-3,-1.65)(-2,-2.2)(-1,-2.42)(0,-2.5)(1,-2.42)(2,-2.2)(3,-1.65)(3.8,0)(3,1.65)
\psline[linecolor=brown](-4.8,-3.3)(4.5,-3.3)
\psline[linecolor=brown,arrows=->,arrowscale=1.5](-4.8,-3.3)(-4.8,2.3)
\rput(-5.2,2.1){$\xi_3$}
\psline[linecolor=red,arrows=->,arrowscale=2](-3,-1.65)(-3,1.65)
\psline[linecolor=brown,linestyle=dashed](-1,-2.42)(-1,-3.3)
\psline[linecolor=red,arrows=->,arrowscale=2](-2,-2.2)(-2,2.2)
\psline[linecolor=red,arrows=->,arrowscale=2](-1,-2.42)(-1,2.42)
\psline[linecolor=red,arrows=->,arrowscale=2](0,-2.5)(0,2.5)
\psline[linecolor=red,arrows=->,arrowscale=2](3,-1.65)(3,1.65)
\psline[linecolor=brown,linestyle=dashed](2,-2.2)(2,-3.3)
\psline[linecolor=red,arrows=->,arrowscale=2](2,-2.2)(2,2.2)
\psline[linecolor=red,arrows=->,arrowscale=2](1,-2.42)(1,2.42)
\psline[linecolor=red,linestyle=dashed,arrows=->,arrowscale=2](-1,-1.9)(-0.6,-1.4)(0.3,-1.5)(1.5,-1)
\psline[linecolor=red,linestyle=dashed](1.5,-1)(2.3,0.5)(2,1.4)
\rput(-1,-3.6){$(\xi_1,\xi_2)$}
\rput(2,-3.65){$\s(\xi_1,\xi_2)$}
\rput(4,2.3){$f^+(\xi_1,\xi_2)$}
\rput(4,-2.2){$f^-(\xi_1,\xi_2)$}
\psdots[dotsize=3pt](-3.8,0)(3.8,0)(-1,-3.3)(2,-3.3)
\rput(-3.4,-2.2){\Large $\pl C^-_v$}
\rput(-3.6,2){\Large $\pl C^+_v$}
\rput(-2.7,1){\Large $v$}
}
\end{picture}
\caption{Restriction of the phase space to the subspace $v =$ const.}
\label{fig Cavalieri}
\end{figure}
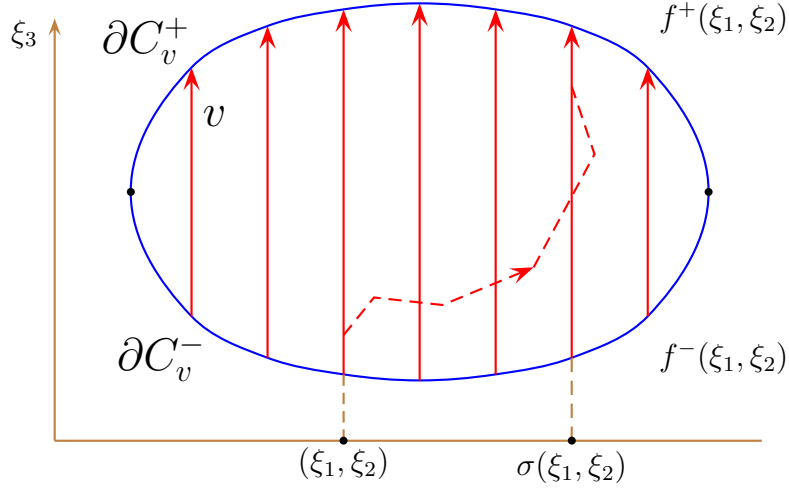

The length $\tau(\xi,v)$ of the billiard trajectory starting at $\xi = (\xi_1, \xi_2, f^-(\xi_1,\xi_2))$ does not exceed the distance between the initial and final points of the trajectory, $(\xi_1, \xi_2, f^-(\xi_1,\xi_2))$ and $(\s(\xi_1,\xi_2), f^+(\s(\xi_1,\xi_2)))$, therefore we obtain the estimate {
$$\tau(\xi,v) \ge \sqrt{|\s(\xi_1,\xi_2)-(\xi_1,\xi_2)|^2 + (f^+(\s(\xi_1,\xi_2)) - f^-(\xi_1,\xi_2))^2} \ge f^+(\s(\xi_1,\xi_2)) - f^-(\xi_1,\xi_2),$$
and thus,}
$$
\int_{\pl C_v^-} \tau(\xi,v)\, (-n(\xi) \cdot v)\, d\xi \ge \int_\Om (f^+(\s(\xi_1,\xi_2)) - f^-(\xi_1,\xi_2))\, d\xi_1\, d\xi_2 =
$$
\beq\label{5}
= \int_\Om f^+(\xi_1,\xi_2)\, d\xi_1\, d\xi_2 - \int_\Om f^-(\xi_1,\xi_2)\, d\xi_1\, d\xi_2.
\eeq
In the last equality the measure preserving property of $\s$ was used. Note also that for some values of $v$ { (and therefore for their neighborhoods)} the inequality in (\ref{5}) is strict.

On the other hand, the length of the trajectory corresponding to $B = \emptyset$ equals $\tau_0(\xi,v) = f^+(\xi_1,\xi_2) - f^-(\xi_1,\xi_2)$, therefore
\beq\label{4}
\int_{\pl C_v^-} \tau_0(\xi,v)\, (-n(\xi) \cdot v)\, d\xi =  \int_\Om (f^+(\xi_1,\xi_2) - f^-(\xi_1,\xi_2))\, d\xi_1\, d\xi_2 \le \int_{\pl C_v^-} \tau(\xi,v)\, (-n(\xi) \cdot v) d\xi.
\eeq
Here again for some values of $v$ the inequality is strict.
Integrating both parts in (\ref{4}) over $v$, we get the phase volume $4\pi |C|$ in the left hand side, and the reachable phase volume (which is less or equal than $4\pi |C \setminus B|$) in the right hand side. Thus, we get
$$
4\pi |C| < 4\pi |C \setminus B|,
$$
which is a contradiction.
\end{proof}

\begin{zam}
\rm Literally repeating the proof for { piecewise smooth} surfaces (which, in contrast to bodies, have zero volume), one concludes that there are no surfaces invisible (or having zero resistance) in all directions.
\end{zam}

\section*{Acknowledgements}

This work was partly supported by the { Center for Research and Development in Mathematics and Applications (CIDMA)} from the ''{\it Funda\c{c}\~{a}o para a Ci\^{e}ncia e a Tecnologia}'' (FCT), cofinanced by the European Community Fund FEDER/POCTI, and by the FCT research project PTDC/MAT/72840/2006.


\begin{thebibliography}{99}



\bibitem{0-resist}
A. Aleksenko and A. Plakhov. {\it Bodies of zero resistance and bodies invisible in one direction}. Nonlinearity {\bf 22}, 1247-1258 (2009).

\bibitem{BB}
{ D. Bucur and G. Buttazzo, {\it Variational Methods in Shape Optimization Problems}.\, Birkh\"auser (2005).}

\bibitem{BK}
{ G. Buttazzo and B. Kawohl. \textit{On Newton's problem of minimal resistance}. Math. Intell. {\bf 15}, 7-12 (1993).}

\bibitem{NYTimes2007} K. Chang. {\it Light Fantastic: Flirting With Invisibility}. The New York Times, June 12 2007.

\bibitem{CL1}
{ M. Comte and T. Lachand-Robert. \textit{Newton's problem of the body of minimal resistance under a single-impact assumption}. Calc. Var. Partial Differ. Equ. {\bf 12}, 173-211 (2001).}

\bibitem{CNN} CNN, {\it Science reveals secters of invisibility}. CNN.com, August 9, 2006.

\bibitem{DailyMail} Daily Mail Reporter, {\it Invisibility cloak a step closer as scientists bend light `the wrong way'}. Daily Mail, August 11, 2008.

\bibitem{CloakErgin}
T. Ergin, N. Stenger, P. Brenner, J. B. Pendry and M. Wegener. {\it Three-Dimensional Invisibility Cloak
at Optical Wavelengths}. Science {\bf 328}, 337--339 (2010).

\bibitem{Dev}
{ A. J. Devaney. {\it Nonuniqueness in the inverse scattering problem.} J. Math. Phys. {\bf 19}, 1526-1531 (1978).}

\bibitem{UTokyo} M. Inami, N. Kawakami, S. Tachi. {\it Optical Camouflage Using Retro-Reflective Projection Technology}
Proceedings of the 2nd IEEE/ACM International Symposium on Mixed and Augmented Reality. p. 348 (2003).

\bibitem{LO}
{ T. Lachand-Robert and E. Oudet. \textit{Minimizing within convex bodies using a convex hull method}.\, SIAM J. Optim. {\bf 16}, 368--379 (2006).}

\bibitem{LP1}
{ T. Lachand-Robert and M.~A. Peletier. \textit{Newton's problem of the body of minimal resistance in the class of convex developable functions}. Math. Nachr. {\bf 226}, 153-176 (2001).}

\bibitem{N}
I. Newton,\, {\it Philosophiae naturalis principia mathematica}.\, 1687.

\bibitem{RMS_review}
{ A Plakhov. {\it Scattering in billiards and problems of Newtonian aerodynamics}. Russ. Math. Surv. {\bf 64}, 873–938 (2009).}

\bibitem{ESAIM}
{ A Plakhov and A Aleksenko. {\it The problem of the body of revolution of minimal resistance}. ESAIM Control Optim. Calc. Var. {\bf 16}, 206-220 (2010).}

\bibitem{CloakSchurig}
D. Schurig, J. J. Mock, B. J. Justice, S. A. Cummer, J. B. Pendry, A. F. Starr and D. R. Smith {\it Metamaterial Electromagnetic Cloak
at Microwave Frequencies}, Science {\bf 314}, 977 (2006).

\bibitem{CloakValentine}
J. Valentine, S. Zhang, T. Zentgraf, E. Ulin-Avila, D. A. Genov, G. Bartal and  X. Zhang.
{\it Three-dimensional optical metamaterial with a negative refractive index}. Nature, {\bf 455} (2008).


\bibitem{Veselago}
V.G. Veselago. {\it The electrodynamics of substances with simultaneously negative values of $\epsilon$ and $\mu$}
Sov. Phys. Usp. 10 (4) (1968), 509–14.

\bibitem{Wolf}
{ E. Wolf and T. Habashy. {\it Invisible bodies and uniqueness of the inverse scattering problem.} J. Modern Optics {\bf 40}, 785-792 (1993).}


\end{thebibliography}
\end{document}